\documentclass{llncs}
\usepackage[left=3.2cm,right=3.2cm,top=2.5cm,bottom=2.5cm]{geometry}
\usepackage{xcolor}
\usepackage{amssymb, amsmath}
\usepackage{enumerate}
\usepackage{algorithmic}
\usepackage{authblk, pstricks}

\newtheorem{Proposition}[definition]{Proposition}

\newtheorem{observation}[theorem]{Observation}
\newcommand{\boxli}{\pi}
\newcommand{\boxlistar}{\boxli^{\star}}
\newcommand{\boxi}{\operatorname{boxicity}}

\newcommand{\leftend}{l}
\newcommand{\rightend}{r}

\newcommand{\floor}[1]{\left\lfloor #1 \right\rfloor}
\newcommand{\ceil}[1]{\left\lceil #1 \right\rceil}
\newcommand{\ilog}{\operatorname{log^{\star}}}
\newcommand{\order}[1]{O\left( #1 \right)}
\newcommand{\orderatleast}[1]{\Omega\left( #1 \right)}

\def\tends{\rightarrow}
\def\into{\rightarrow}
\def\N{\mathbb{N}}
\def\R{\mathbb{R}}

\def\G{\mathcal{G}}
\def\F{\mathcal{F}}
\def\E{\mathcal{E}}

\begin{document}
%\date{}
%\author[1]{Manu~Basavaraju}
%\author[2]{L.~Sunil~Chandran}
%\author[3]{Martin~Charles~Golumbic}
%\author[3]{Rogers~Mathew\footnote{Supported by VATAT Postdoctoral Fellowship}}
%\author[3]{Deepak~Rajendraprasad\footnote{Supported by VATAT Postdoctoral Fellowship}}

%\affil[1]{
%	University of Bergen, \authorcr
%	Postboks 7800, NO-5020 Bergen. \authorcr
%	manu.basavaraju@ii.uib.no
%}
%\affil[2]{
%	Department of Computer Science and Automation, \authorcr 
%	Indian Institute of Science,  
%    	Bangalore, India - 560012. \authorcr
%	sunil@csa.iisc.ernet.in
%}
%\affil[3]
%{
%	Department of Computer Science,  
%	Caesarea Rothschild Institute, \authorcr
%	University of Haifa, 
%	31905 Haifa, Israel. \authorcr
%	\{golumbic, rogersmathew, deepakmail\}@gmail.com
%}
\title{Boxicity and separation dimension}
\author{Manu~Basavaraju\inst{1} \and L. Sunil Chandran\inst{2} \and Martin~Charles~Golumbic\inst{3} \and  Rogers Mathew\inst{3}
%\thanks{Supported by VATAT Postdoctoral Fellowship}
 \and Deepak~Rajendraprasad\inst{3}
%\thanks{Supported by VATAT Postdoctoral Fellowship}
}
\institute{University of Bergen, Postboks 7800, NO-5020 Bergen. \\
	\email{manu.basavaraju@ii.uib.no}
\and
	Department of Computer Science and Automation, \\ 
	Indian Institute of Science, Bangalore -- 560012, India.\\
	\email{sunil@csa.iisc.ernet.in}
\and
	Department of Computer Science, Caesarea Rothschild Institute, \\
	University of Haifa, 31905 Haifa, Israel. \\
	\email{golumbic@cs.haifa.ac.il, \{rogersmathew, deepakmail\}@gmail.com}
}
\maketitle
\begin{abstract}
A family $\F$ of permutations of the vertices of a hypergraph $H$ is called {\em pairwise suitable} for $H$ if, for every pair of disjoint edges in $H$, there exists a permutation in $\F$ in which all the vertices in one edge precede those in the other. The cardinality of a smallest such family of permutations for $H$ is called the {\em separation dimension} of $H$ and is denoted by $\boxli(H)$. Equivalently, $\boxli(H)$ is the smallest natural number $k$ so that the vertices of $H$ can be embedded in $\R^k$ such that any two disjoint edges of $H$ can be separated by a hyperplane normal to one of the axes. We show that the separation dimension of a hypergraph $H$ is equal to the {\em boxicity} of the line graph of $H$. This connection helps us in borrowing results and techniques from the extensive literature on boxicity to study the concept of separation dimension. 

% For a rank-$r$ hypergraps$H$ on $n$ vertices, we show an upper bound on $\boxli(H)$ in terms of $n$ and $r$ and another upper bound in terms of the maximum degree and $r$. For a graph $G$ on $n$ vertices, we demonstrate several upper bounds for $\boxli(G)$ in terms of invariants of $G$ like $n$, maximum degree, degeneracy, pathwidth, and acyclic chromatic number of $G$. We also show tight bounds on the separation dimension for planar graphs, hypercubes and fully subdivided graphs. We then show a particular type of obstruction which, if present in $G$, prohibits $\boxli(G)$ from falling too low. This is used to show that some of the above upper bounds are tight up to constant factors, and also to derive a lower bound for random graphs. The technique is then extended to give lower bounds for hypergraphs. We conclude with a few interesting open questions.

\vspace{1ex}
\end{abstract}

%%%%% Abstract ends %%%%%%%%%%%%%%%%%%%%%%%%%%%%%%%%%%%%%%%%%%%%%%%%%%%%%%%%%%%%%%%%%%%%%%%%%%%%%%%

\section{Introduction}

Let $\sigma:U \into [n]$ be a permutation of  elements of an $n$-set $U$. For two disjoint subsets $A,B$ of $U$, we say  $A \prec_{\sigma} B$ when every element of $A$ precedes every element of $B$ in $\sigma$, i.e., $\sigma(a) < \sigma(b), \forall (a,b) \in A \times B$. Otherwise, we say $A \nprec_{\sigma} B$. We say that $\sigma$ {\em separates} $A$ and $B$ if either $A \prec_{\sigma} B$ or $B \prec_{\sigma} A$. We use $a \prec_{\sigma} b$ to denote $\{a\} \prec_{\sigma} \{b\}$. For two subsets $A, B$ of $U$, we say $A \preceq_{\sigma} B$ when $A \setminus B \prec_{\sigma} A\cap B \prec_{\sigma} B \setminus A$. 

In this paper, we introduce and study a notion called \emph{pairwise suitable family of permutations} for a hypergraph $H$ and the \emph{separation dimension} of $H$.

\begin{definition}
\label{definitionPairwiseSuitable}
A family $\F$ of permutations of $V(H)$ is \emph{pairwise suitable} for a hypergraph $H$ if, for every two disjoint edges $e,f \in E(H)$, there exists a permutation $\sigma \in \F$  which separates $e$ and $f$. The cardinality of a smallest family of permutations that is pairwise suitable for $H$ is called the {\em separation dimension} of $H$ and is denoted by $\boxli(H)$. 
\end{definition}

A family $\F = \{\sigma_1, \ldots, \sigma_k\}$ of permutations of a set $V$ can be seen as an embedding of $V$ into $\R^k$ with the $i$-th coordinate of $v \in V$ being the rank of $v$ in the $\sigma_i$. Similarly, given any embedding of $V$ in $\R^k$, we can construct $k$ permutations by projecting the points onto each of the $k$ axes and then reading them along the axis, breaking the ties arbitrarily. From this, it is easy to see that $\boxli(H)$ is the smallest natural number $k$ so that the vertices of $H$ can be embedded into $\R^k$ such that any two disjoint edges of $H$ can be separated by a hyperplane normal to one of the axes. This motivates us to call such an embedding a {\em separating embedding} of $H$ and $\boxli(H)$ the {\em separation dimension} of $H$.

The study of similar families of permutations dates back to the work of Ben Dushnik in 1950 where he introduced the notion of \emph{$k$-suitability} \cite{dushnik}. A family $\F$ of permutations of $[n]$ is \emph{$k$-suitable} if, for every $k$-set $A \subseteq [n]$ and for every $a \in A$, there exists a $\sigma \in \mathcal{F}$ such that $A \preceq_{\sigma} \{a\}$. Let $N(n,k)$ denote the cardinality of a smallest family of permutations that is $k$-suitable for $[n]$. In 1972, Spencer \cite{scramble} proved that $\log \log n \leq N(n,3) \leq N(n,k) \leq k2^k\log \log n$. He also showed that $N(n,3) < \log \log n + \frac{1}{2} \log \log \log n + \log (\sqrt{2}\boxli) + o(1)$. Fishburn and Trotter, in 1992, defined the \emph{dimension} of a hypergraph on the vertex set $[n]$ to be the minimum size of a family $\F$ of permutations of $[n]$ such that every edge of the hypergraph is an intersection of \emph{initial segments} of $\F$ \cite{FishburnTrotter1992}. It is easy to see that an edge $e$ is an intersection of initial segments of $\F$ if and only if for every $v \in [n] \setminus e$, there exists a permutation $\sigma \in \F$ such that $e \prec_{\sigma} \{v\}$. F\"{u}redi, in 1996, studied the notion of \emph{$3$-mixing} family of permutations \cite{Furedi1996}. A family $\F$ of permutations of $[n]$ is called $3$-mixing if for every $3$-set $\{a, b, c\} \subseteq [n]$ and a designated element $a$ in that set, one of the permutations in $\F$ places the element $a$ between $b$ and $c$. It is clear that $a$ is between $b$ and $c$ in a permutation $\sigma$ if and only if $\{a,b\} \preceq_{\sigma} \{a,c\}$ or $\{a.c\} \preceq_{\sigma} \{a,b\}$. Such families of permutations with small sizes have found applications in showing upper bounds for many combinatorial parameters like poset dimension \cite{kierstead1996order}, product dimension \cite{FurediPrague}, boxicity \cite{RogSunSiv} etc.

The notion of separation dimension introduced here seems so natural but, to the best of our knowledge, has not been studied in this generality before. The authors of \cite{chee2013sequence} provide suggested applications motivating the study of permutation  covering and separation problems on event sequencing of tasks. Apart from that, a major motivation for us to study this notion of separation is its interesting connection with a certain well studied geometric representation of graphs.  In fact, we show that $\boxli(H)$ is same as the \emph{boxicity} of the intersection graph of the edge set of $H$, i.e., the line graph of $H$.
  
An axis-parallel $k$-dimensional box or a \emph{$k$-box} is  a Cartesian  product $R_1 \times  \cdots \times R_k$, where each  $R_i$ is a closed interval on the real line. For example, a line segment lying parallel to the $X$ axis is a $1$-box, a rectangle with its sides parallel to the $X$ and $Y$ axes is a $2$-box, a rectangular cuboid with its sides parallel to the $X$, $Y$, and $Z$ axes is a $3$-box and so on. A {\em box representation} of a graph $G$ is a geometric representation of $G$ using axis-parallel boxes as follows. 

\begin{definition}
\label{definitionBoxicity}
The \emph{$k$-box representation} of a graph $G$ is a function $f$ that maps each vertex in $G$ to a $k$-box in $\mathbb{R}^k$ such that, for all vertices $u,v$ in $G$, the pair $\{u,v\}$ is an edge if and only if $f(u)$ intersects $f(v)$. The \emph{boxicity} of a graph $G$, denoted by $\boxi(G)$, is the minimum positive integer $k$ such that $G$ has a $k$-box representation. 
\end{definition}

Box representation is a generalisation of interval representation of \emph{interval graphs} (intersection graphs of closed intervals on the real line). From the definition of boxicity, it is easy to see that interval graphs are precisely the graphs with boxicity $1$.  The concept of boxicity was introduced by F.S. Roberts in 1969 \cite{Roberts}. He showed that every graph on $n$ vertices has an $\floor{ n/2  }$-box representation. The $n$-vertex graph whose complement is a perfect matching is an example of a graph whose boxicity is equal to $n/2$. Upper bounds for boxicity in terms of other graph parameters like maximum degree, treewidth, minimum vertex cover, degeneracy etc. are available in literature. Adiga, Bhowmick, and Chandran showed that the boxicity of a graph with maximum degree $\Delta$ is $O(\Delta \log^2 \Delta)$ \cite{DiptAdiga}. Chandran and Sivadasan proved that boxicity of a graph with treewidth $t$ is at most $t+2$ \cite{CN05}. It was shown by Adiga, Chandran and Mathew that the boxicity of a $k$-degenerate graph on $n$ vertices is $O(k \log n)$ \cite{RogSunAbh}. Boxicity is also studied in relation with other dimensional parameters of graphs like partial order dimension and threshold dimension \cite{DiptAdiga,Yan1}. Studies on box representations of special graph classes too are available in abundance. Scheinerman showed that every outerplanar graph has a $2$-box representation \cite{Scheiner} while Thomassen showed that every planar graph has a $3$-box representation \cite{Thoma1}. Results on boxicity of series-parallel graphs \cite{CRB1}, Halin graphs \cite{halinbox}, chordal graphs, AT-free graphs, permutation graphs \cite{CN05}, circular arc graphs \cite{Dipt}, chordal bipartite graphs \cite{SunMatRog} etc. can be found in literature. Here we are interested in boxicity of the line graph of hypergraphs.

\begin{definition}
\label{definitionLineGraph}
The \emph{line graph} of a hypergraph $H$, denoted by $L(H)$, is the graph with vertex set $V(L(H)) = E(H)$ and edge set $E(L(H)) = \{\{e,f\} : e, f \in E(H), e \cap f \neq \emptyset \}$.
\end{definition}
  
For the line graph of a graph $G$ with maximum degree $\Delta$, it was shown by Chandran, Mathew and Sivadasan that its boxicity is $\order{\Delta \log\log \Delta}$ \cite{RogSunSiv}. It was in their attempt to improve this result that the authors stumbled upon pairwise suitable family of permutations and its relation with the boxicity of the line graph of $G$. In an arxiv preprint version of this paper available at \cite{sepDim}, we improve the upper bound for boxicity of the line graph of $G$ to $2^{9 \ilog \Delta} \Delta$, where $\ilog \Delta$ denotes the iterated logarithm of $\Delta$ to the base $2$, i.e. the number of times the logarithm function (to the base $2$) has to be applied so that the result is less than or equal to $1$. Bounds for separation dimension of a graph based on its treewidth, degeneracy etc. are also established in this version. 

\subsection{Outline of the paper}

The remainder of this paper is organised as follows. A brief note on some standard terms and notations used throughout this paper is given in Section \ref{sectionNotation}. Section \ref{sectionConnection} demonstrates the equivalence of separation dimension of a hypergraph $H$ and boxicity of the line graph of $H$. In Section \ref{CharacterizingDimOneSection}, we characterize graphs of separation dimension $1$. Using a probabilistic argument, in Section \ref{subsection:SizeGraph}, we prove a tight (up to constants) upper bound for separation dimension of a graph based on its size. Section \ref{subsection:Acyclic} relates separation dimension with acyclic chromatic number. In Section \ref{subsection:planar}, using  Schnyder's celebrated result on planar drawing, we show that the separation dimension of a planar graph is at most $3$. This bound is the best possible as we know of series-parallel graphs (that are subclasses of planar graphs) of separation dimension $3$. In Section \ref{subsection:LowerBound}, we prove the theorem that yields a non-trivial lower bound to the separation dimension of a graph. This theorem and its corollaries are used in establishing the tightness of the upper bounds proved. Moreover, the theorem is used to prove a lower bound for the separation dimension of a random graph in Section \ref{sectionRandomGraphs}. 

Once again, in Section \ref{sectionBoxliSizeHyper}, we use a probabilistic argument to show an upper bound on the separation dimension of a rank-$r$ hypergraph based on its size. This is followed by an upper bound based on maximum degree in Section \ref{Section:HyperGraphMaxDeg}. We get this upper bound as a consequence of a non-trivial result in  the area of boxicity. In Section \ref{sectionLowerBoundHypergraphs}, we prove a lower bound on the separation dimension of a complete $r$-uniform hypergraph by extending the lower bounding technique used in the context of graphs. 
%All the upper and lower bounds related to separation dimension of hypergraphs are stated and proved in Sections \ref{sectionBoxliSizeHyper}, \ref{Section:HyperGraphMaxDeg}, and \ref{sectionLowerBoundHypergraphs}. %The tightness of the upper bounds, where we know them, are mentioned alongside the bound but their proofs and discussion are postponed till the subsequent section (Section \ref{subsection:LowerBound}). 
Finally, in Section \ref{sectionOpenProblems}, we conclude with a discussion of a few open problems that we find interesting.

% \tableofcontents

%%%%% Outline subsection ends. %%%%%%%%%%%%%%%%%%%%%%%%%%%%%%%%%%%%%%%%%%%%%%%%%%%%%%%%%%%%%%%%%%%%

\subsection{Notational note}
\label{sectionNotation}

A {\em hypergraph} $H$ is a pair $(V, E)$ where $V$, called the {\em vertex set}, is any set and $E$, called the {\em edge set}, is a collection of subsets of $V$. The vertex set and edge set of a hypergraph $H$ are denoted respectively by $V(H)$ and $E(H)$. The {\em rank} of a hypergraph $H$ is $\max_{e \in E(H)}|e|$ and $H$ is called {\em $k$-uniform} if $|e| = k, \forall e \in E(H)$. The {\em degree} of a vertex $v$ in $H$ is the number of edges of $H$ which contain $v$. The {\em maximum degree} of $H$, denoted as $\Delta(H)$ is the maximum degree over all vertices of $H$. All the hypergraphs considered in this paper are finite. 

A {\em graph} is a $2$-uniform hypergraph. For a graph $G$ and any $S \subseteq V(G)$, the subgraph of $G$ induced by the vertex set $S$ is denoted by $G[S]$. For any $v \in V(G)$, we use $N_G(v)$ to denote the neighbourhood of $v$ in $G$, i.e., $N_G(v) = \{u \in V(G) : \{v,u\} \in E(G)\}$.

A \emph{closed interval} on the real line, denoted as $[i,j]$ where $i,j \in \R$ and $i\leq j$, is the set $\{x\in \R : i\leq x\leq j\}$. Given an interval $X=[i,j]$, define $\leftend(X)=i$ and $\rightend(X)=j$. We say that the closed interval $X$ has \emph{left end-point} $\leftend(X)$ and \emph{right end-point} $\rightend(X)$. For any two intervals $[i_1, j_1], [i_2,j_2]$ on the real line, we say that $[i_1, j_1] < [i_2,j_2]$ if $j_1 < i_2$. % Clearly, $[i_1, j_1] \cap [i_2,j_2] = \emptyset$ if and only if $[i_1, j_1] < [i_2,j_2]$ or $[i_2,j_2] < [i_1, j_1]$.

For any finite positive integer $n$, we shall use $[n]$ to denote the set $\{1,\ldots , n\}$. A permutation of a finite set $V$ is a bijection from $V$ to $[|V|]$.
The logarithm of any positive real number $x$ to the base $2$ and $e$ are respectively denoted by $\log(x)$ and $\ln(x)$. 
%, while $\ilog(x)$ denotes the iterated logarithm of $x$ to the base $2$, i.e. the number of times the logarithm function (to the base $2$) should be applied so that the result is less than or equal to $1$.

%%%%% Notation subsection ends. %%%%%%%%%%%%%%%%%%%%%%%%%%%%%%%%%%%%%%%%%%%%%%%%%%%%%%%%%%%%%%%%%%%
%%%%% Introduction section ends. %%%%%%%%%%%%%%%%%%%%%%%%%%%%%%%%%%%%%%%%%%%%%%%%%%%%%%%%%%%%%%%%%%

\section{Pairwise suitable family of permutations and a box representation}
\label{sectionConnection}

In this section we show that a family of permutations of cardinality $k$ is pairwise suitable for a hypergraph $H$ (Definition \ref{definitionPairwiseSuitable}) if and only if the line graph of $H$ (Definition \ref{definitionLineGraph}) has a $k$-box representation (Definition \ref{definitionBoxicity}). Before we proceed to prove it, let us state an equivalent but more combinatorial definition for boxicity. 

We have already noted that interval graphs are precisely the graphs with boxicity $1$. Given a $k$-box representation of a graph $G$, orthogonally projecting the $k$-boxes to each of the $k$-axes in $\mathbb{R}^k$  gives $k$ families of intervals. Each one of these families can be thought of as an interval representation of some interval graph. Thus we get $k$ interval graphs. It is not difficult to observe that a pair of vertices is  adjacent in $G$ if and only if the pair is adjacent in each of the $k$ interval graphs obtained. The following lemma, due to Roberts \cite{Roberts}, formalises this relation between box representations and interval graphs.

\begin{lemma}[Roberts \cite{Roberts}]
\label{lemmaRoberts}
For every graph $G$, $\boxi(G) \leq k$ if and only if there exist $k$ interval graphs $I_1, \ldots, I_k$, with $V(I_1) = \cdots = V(I_k) = V(G)$ such that $G = I_1 \cap \cdots \cap I_k$.
\end{lemma}
From the above lemma, we get an equivalent definition of boxicity. 

\begin{definition}
\label{definitionBoxicityInterval}
The \emph{boxicity} of a 
graph $G$ is the minimum positive integer $k$ for which there exist $k$ interval graphs
$I_1,\ldots, I_k$ such that $G = I_1 \cap \cdots \cap I_k$.  
\end{definition}

Note that if $G = I_1 \cap \cdots \cap I_k$, then each $I_i$ is a supergraph of $G$. Moreover, for every pair of vertices $u,v \in V(G)$ with $\{u,v\} \notin E(G)$, there exists some $i \in [k]$ such that $\{u,v\} \notin E(I_i)$. Now we are ready to prove the main theorem of this section.

\begin{theorem}
For a hypergraph $H$, $\boxli(H) = \boxi(L(H))$.  
\label{theoremConnectionBoxliPermutation}
\end{theorem}
\begin{proof}
First we show that $\boxli(H) \leq \boxi(L(H))$. Let $\boxi(L(H)) = b$. Then, by Lemma \ref{lemmaRoberts}, there exists a collection of $b$ interval graphs, say $\mathcal{I} = \{I_1, \ldots, I_b\}$, whose intersection is $L(H)$. For each $i \in [b]$, let $f_i$ be an interval representation of $I_i$. For each $u \in V(H)$, let $E_H(u) = \{e \in E(H) : u \in e\}$ be the set of edges of $H$ containing $u$. Consider an $i \in [b]$ and a vertex  $u \in V(H)$. The closed interval $C_i(u) = \bigcap_{e \in E_H(u)} f_i(e)$ is called the {\em clique region} of $u$ in $f_i$. Since any two edges in $E_H(u)$ are adjacent in $L(H)$, the corresponding intervals have non-empty intersection in $f_i$ . By the Helly property of intervals, $C_i(u)$ is non-empty. We define a permutation $\sigma_i$ of $V(H)$ from $f_i$ such that $\forall u,v \in V(G)$, $C_i(u) < C_i(v) \implies u \prec_{\sigma_i} v$. It suffices to prove that $\{\sigma_1, \ldots, \sigma_b\}$ is a family of permutations that is pairwise suitable for $H$.

Consider two disjoint edges $e, e'$ in $H$. Hence $\{e, e' \} \notin E(L(H))$ and since $L(H) = \bigcap_{i=1}^b I_i$, there exists an interval graph, say $I_i \in \mathcal{I}$, such that $\{e, e'\} \notin E(I_i)$, i.e., $f_i(e) \cap f_i(e') = \emptyset$. Without loss of generality, assume $f_i(e) < f_i(e')$. For any $v \in e$ and any $v' \in e'$, since $C_i(v) \subseteq f_i(e)$ and $C_i(v') \subseteq f(e')$, we have $C_i(v) < C_i(v')$, i.e. $v \prec_{\sigma_i} v'$. Hence $e \prec_{\sigma_i} e'$. Thus the family $\{ \sigma_1, \ldots, \sigma_b \}$ of permutations is pairwise suitable for $H$. 

Next we show that $\boxi(L(H)) \leq \boxli(H)$. Let $\boxli(H) = p$ and let $\F = \{\sigma_1, \ldots, \sigma_p\}$ be a pairwise suitable family of permutations for $H$. From each permutation $\sigma_i$, we shall construct an interval graph $I_i$ such that $L(H) = \bigcap_{i=1}^{p}I_i$. Then by Lemma \ref{lemmaRoberts}, $\boxi(L(H)) \leq \boxli(H)$. 

For a given $i \in [p]$, to each edge $e \in E(H)$, we associate the closed interval 
$$f_i(e) = \left[ \min_{v \in e}\sigma_i(v) ~,~ \max_{v \in e}\sigma_i(v) \right],$$ 
and let $I_i$ be the intersection graph of the intervals $f_i(e), e \in E(H)$. Let $e, e' \in V(L(H))$. If $e$ and  $e'$ are adjacent in $L(H)$, let $v \in e \cap e'$. Then $\sigma_i(v) \in f_i(e) \cap f_i(e'),~\forall i \in [p]$. Hence $e$ and $e'$ are adjacent in $I_i$ for every $i \in [p]$. If $e$ and $e'$ are not adjacent in $L(H)$, then there is a permutation $\sigma_i \in \F$ such that either $e \prec_{\sigma_i} e'$ or $e' \prec_{\sigma_i} e$. Hence by construction $f_i(e) \cap f_i(e') = \emptyset$ and so $e$ and $e'$ are not adjacent in $I_i$. This completes the proof.
\qed
\end{proof}

It is the discovery of this intriguing connection that aroused our interest in the study of pairwise suitable families 
of permutations. This immediately makes applicable every result in the area of boxicity to separation dimension. For 
 example, any hypergraph with $m$ edges can be separated in $\R^{\floor{m/2}}$; for every $m \in \N$, there exist 
 hypergraphs with $m$ edges which cannot be separated in any proper subspace of $\R^{\floor{m/2}}$; every hypergraph 
 whose line graph is planar can be separated in $\R^3$; every hypergraph whose line graph has a treewidth at most $t$ 
 can be separated in $\R^{t+2}$; hypergraphs separable in $\R^1$ are precisely those whose line graphs are interval 
 graphs and so on. Further, algorithmic and hardness results from boxicity carry over to separation dimension since 
 constructing the line graph of a hypergraph can be done in quadratic time. We just mention two of them. Deciding if 
 the separation dimension is at most $k$ is NP-Complete for every $k \geq 2$ \cite{Coz,Kratochvil} and unless NP = 
 ZPP, for any $\epsilon >0$, there does not exist a polynomial time algorithm to approximate the separation dimension 
 of a hypergraph within a factor of $m^{1/2 -\epsilon}$ where $m = |E(H)|$ \cite{DiptAdigaHardness} \footnote{A 
 recent  paper shows that the inapproximability factor can be improved to $m^{1 - \epsilon}$, which is essentially 
 tight \cite{chalermsookgraph}.}. In this work, we have tried to find bounds on the separation dimension of a 
 hypergraph in terms of natural invariants of the hypergraph like maximum degree, rank etc. The next two results are 
 for rank-$r$ hypergraphs. 
%%%%% Connection section ends. %%%%%%%%%%%%%%%%%%%%%%%%%%%%%%%%%%%%%%%%%%%%%%%%%%%%%%%%%%%%%%%%%%%%
\section{Separation dimension of graphs}
\label{section:ResultsOnGraphs}
\subsection{Characterizing graphs of separation dimension $1$}
\label{CharacterizingDimOneSection}
Before we characterize graphs of separation dimension at most $1$, let us take note of this easy observation that follows directly from the definition. 

\begin{observation}
\label{observationPiMonotonicity}
$\boxli(G)$ is a monotone increasing property, i.e., $\boxli(G') \leq \boxli(G)$ for every subgraph $G'$ of $G$.  
\end{observation}

``When is $\pi(G)=0$?''  Clearly, if $\boxli(G) = 0$, then $G$ may have at most one non-trivial connected component and every pair of edges must share an endpoint.  The following is a simple exercise answering the question:

\begin{Proposition}
   For a graph $G$, we have $\pi(G)=0$ if and only if $G$ is either a star or a triangle
   plus an unlimited number of isolated vertices.
\end{Proposition}

A \emph{caterpillar} is a tree consisting of a chordless path $[v_1,v_2,\ldots,v_k]$
called the \emph{spine}, plus an unlimited number of pendant vertices.
A \emph{caterpillar with single humps} is formed from a caterpillar by adding at most one
new vertex $x_i$ adjacent to $v_i$ and $v_{i+1}$ for every $i=1,\ldots, k-1$.
Without loss of generality, we may assume that the first and last vertex of the spine have
no pendent vertices (i.e., the spine is longest possible.)
The \emph{diamond}, denoted here by $D$, is the graph with 4 vertices and 5 edges;
the 3-net $N_3$ consists of a triangle with a pendant vertex attached to each of its vertices; the graph $T_2$ is the tree with 6 edges $\{cx,cy,cz,xx',yy',zz'\}$; and the graph $C_k$ ($k\geq 4$) denotes the cycle of size $k$.

\begin{theorem}
\label{prop:caterpillar}
   Let $G$ be a graph. The following conditions are equivalent:
   \begin{enumerate}
     \item [(i)] $\pi(G) \leq 1$,
     \item [(ii)] $G$ is a disjoint union of caterpillars with single humps, 
%\footnote{\RogersDeepak{Should we say that $G$ is a disjoint union of caterpillars with single humps?}},
     \item [(iii)] $G$ has no partial subgraph $C_k$ ($k\geq 4$), $N_3$ or $T_2$,
     \item [(iv)] $G$ is a chordal graph with no induced subgraph
            $D$, $K_4$, $T_2$, $N_3$, $G_1$, $G_2$ or $G_3$, where $G_1 = T_2 \cup \{cx'\}$, $G_2 = G_1 \cup \{cy'\}$ and $G_3 = G_2 \cup \{cz'\}$,
     \item [(v)] The line graph $L(G)$ is an interval graph.
   \end{enumerate}
\end{theorem}

\begin{proof} $(i)\Rightarrow(iii)$.
    The graphs $C_k$ ($k\geq 4$), $N_3$ and $T_2$ all have separation dimension 2,
    so any graph $G$ containing one of these graphs as partial subgraphs would have $\pi(G)\geq 2$ due to Observation~\ref{observationPiMonotonicity}.
    Note that each of these is a minimal forbidden subgraph, also they are all needed in the characterization.

 $(iii)\Rightarrow(ii)$. 
 Assume that $G$ is a connected graph; otherwise we can apply the same arguments to each of its components. Suppose $G$ has no partial subgraphs $C_k$ ($k\geq 4$), $N_3$ or $T_2$,
 so, in particular, $G$ is a chordal graph and has a symplicial vertex, say $a$,
 whose neighborhood is a clique. If the degree of $a$ is greater than or equal to 3, 
 then $G$ would have a clique of size 4 which has $C_4$ as a partial subgraph.
 So the degree of $a$ must be 1 or 2.
 
 We proceed by induction: assuming that $G / \{a\}$ is a caterpillar with single humps,
 having spine $v_1, v_2,\ldots,v_k$ such that neither $v_1$ nor $v_k$ has a pendent
 vertex. 
 
 \emph{Case 1.} Suppose $deg(a)=1$ in $G$, and let $b$ be the neighbor of $a$.  
  If $b$ is on the spine, then $a$ is a pendent to $b$ and we are done.
 
  If $b$ is a hump vertex adjacent to say $v_i$ and $v_{i+1}$,
    then the edge $ab$ implies that either $i=1$ or $i+1=k$, since $G$ has no $N_3$.
    If $i=1$ then $ab,bv_{i+1}$ moves to the spine replacing $v_i v_{i+1}$,
        making $v_i$ a hump vertex.
    If $i+1=k$ then $v_i b, ba$ moves to the spine replacing $v_i v_{i+1}$,
       making $v_{i+1}$ a hump vertex.
       
  Finally, if $b$ is a pendent vertex, the fact that $G$ has no partial subgraph $T_2$ proves the following claim:

 \emph{ Claim 1:  If $b$ is a pendent vertex adjacent to $v_j$ 
    and $a$ is adjacent to $b$, then (i) at least one of $v_{j-1}$ or $v_{j+1}$
    is an endpoint of the spine, and (ii) if one of $v_{j-1} v_j$ or $v_j v_{j+1}$
    has a hump vertex, then the other has no hump and is an end of the spine.}

 Without loss of generality, assume that $v_{j-1} v_j$ is an end of the spine (i.e., $j=2$)
 and has no hump vertex, the case of $v_j v_{j+1}$ being similar. 
 Then, $ab,bv_{j}$ moves to the spine replacing $v_{j-1} v_j$
 which becomes a pendant edge.

 \emph{Case 2.} Suppose $deg(a)=2$ in $G$, and let $b,c$ be the neighbors of $a$.  
 Since $a$ is simplicial, $\{a,b,c\}$ form a triangle.  
 
 If $b$ is a hump vertex, then $c$ is on the spine and $\{b,c,v_i\}$ form a triangle with another spine vertex $v_i$; but then $\{a,b,c,v_i\}$ would form a diamond $D$ contradicting the assumption that $G$ has no partial subgraph $C_4$.
 
 If $b$ is a pendant vertex, then $c$ is on the spine (say, $c=v_{j}$)
 and has spine neighbors $v_{j-1}$ and $v_{j+1}$. 
 By Claim 1, we may assume without loss of generality, that $v_{j-1} v_j$ is an end of the spine and has no hump vertex, the case of $v_j v_{j+1}$ being similar. 
 Then, $av_{j}$ moves to the spine replacing $v_{j-1} v_j$,
 making $b$ a single hump and $v_{j-1}$ a pendant vertex.
  
 Finally, if $bc$ is on the spine, then $a$ will be a single hump, for otherwise, if $bc$ already had a hump, then together with $a$ would form a diamond $D$ contradicting the assumption that $G$ has no partial subgraph $C_4$. 
 
 Hence, we prove that $G$ is a caterpillar with single humps. 
%Now in the case when $G$ is not a connected graph, we use the above arguments to show that each of its components is a caterpillar with single humps. 
 
  $(ii)\Rightarrow(i)$. An edge separating permutation $\sigma$ for a caterpillar with single    humps can be obtained by listing the vertices on the spine $v_1, v_2,\ldots,v_k$
  as they appear from left to right and inserting immediately after $v_i$ all of its
  pendant neighbors followed by the common hump neighbor with $v_{i+1}$ if it exists.
  
  $(i)\Leftrightarrow(v)$ follows from Theorem \ref{theoremConnectionBoxliPermutation}, and 
  $(iii)\Leftrightarrow(iv)$ is a straightforward exercise.

\qed \end{proof}

%\begin{Remark}
%\label{rem:algorithm-caterpillar}
The proof of Theorem \ref{prop:caterpillar} suggests a linear time algorithm for
recognizing whether a graph $G$ has separation dimension 1 and constructing its 
representation as a caterpillar with single humps: 
    (1) Using either Lexicographic Breadth First Search or Maximum Cardinality Search,
        obtain an ordering of the vertices $a_1, a_2,\ldots,a_n$ (but do not bother to
        test whether it is a perfect elimination ordering\footnote{If $G$ is chordal, 
         any LexBFS or MCS ordering will be a perfect elimination ordering, but testing whether each $v_i$ has exactly one forward neighbor or two connected forward neighbors will be enough.};
    (2) Starting with $a_n$ and proceeding in reverse order, follow the rules in the proof 
        of $(iii)\Rightarrow(ii)$ to construct the spine, pendant vertices and the humps. 
    If either (1) or (2) fails, then $\pi(G) > 1$.

%\section{Upper bounds}
%\label{sectionUpperBounds}

% Now we show an upper bound on $\boxli(G)$ in terms of $|V(G)|$.

%%%%%% Subsections of upper bounds begins. %%%%%%%%%%%%%%%%%%%%%%%%%%%%%%%%%%%%%%%%%%%%%%%%%%%%%%%%

\subsection{Separation dimension and size of a graph}
\label{subsection:SizeGraph}

For graphs, sometimes we work with a notion of suitability that is stronger than the pairwise suitability of Definition \ref{definitionPairwiseSuitable}. This will facilitate easy proofs for some results to come later in this article.

\begin{definition}
\label{definition3Mixing}
For a graph $G$, a family $\F$ of permutations of $G$ is \emph{$3$-mixing} if, for every two adjacent edges $\{a,b\}, \{a,c\} \in E(G)$, there exists a permutation $\sigma \in \F$ such that either $b \prec_{\sigma} a \prec_{\sigma} c$ or $c \prec_{\sigma} a \prec_{\sigma} b$. 
\end{definition}

Notice that a family of permutations $\mathcal{F}$ of $V(G)$ is pairwise suitable and $3$-mixing for $G$ if, for every two edges $e,f \in E(G)$, there exists a permutation $\sigma \in \mathcal{F}$ such that either $e \preceq_{\sigma} f$ or $f \preceq_{\sigma} e$.  Let $\boxlistar(G)$ denote the cardinality of a smallest family of permutations that is pairwise suitable and $3$-mixing for $G$. From their definitions, $\boxli(G) \leq \boxlistar(G)$.

We begin with the following two straightforward observations. 

\begin{observation}
\label{observationPiStarMonotonicity}
$\boxlistar(G)$ is a monotone increasing property. 
%, i.e., $\boxli(G') \leq \boxli(G)$ and $\boxlistar(G') \leq \boxlistar(G)$ for every subgraph $G'$ of $G$.  
\end{observation}

\begin{observation}
\label{observationDisjointComponents}
Let $G_1, \ldots, G_r$ be a collection of disjoint components that form a graph $G$, i.e, $V(G) = \biguplus_{i=1}^r V(G_i)$ and $E(G) = \biguplus_{i=1}^r E(G_i)$. If $\boxli(G) \geq 1$ for some $i \in [r]$, then $\boxli(G) = \max_{i \in [r]} \boxli(G_i)$. 
\end{observation}

 A nontrivial generalisation of Observation \ref{observationDisjointComponents}, when there are edges across the parts, is given later in Lemma \ref{lemmaMaxPairsParts}.

%Using a simple probabilistic argument one can prove the following theorem. 
\begin{theorem}
\label{theoremBoxliSize}
For a graph $G$ on $n$ vertices,  $\boxli(G) \leq \boxlistar(G) \leq 6.84 \log n$. 
\end{theorem}
\begin{proof}
From the definitions of $\boxli(G)$ and $\boxlistar(G)$ and Observation \ref{observationPiStarMonotonicity}, we have $\boxli(G) \leq \boxlistar(G) \leq \boxlistar(K_n)$, where $K_n$ denotes the complete graph on $n$ vertices. Here we prove that  $\boxlistar(K_n) \leq 6.84 \log n$. 

Choose $r$ permutations, $\sigma_1, \ldots, \sigma_r$, independently and uniformly at random from the $n!$ distinct permutations of $[n]$. Let $e$, $f$ be two distinct edges of $K_n$. The probability that $e \preceq_{\sigma_i} f$ is $1/6$ for each $i \in [r]$. ($4$ out of $4!$ outcomes are favourable when $e$ and $f$ are non-adjacent and $1$ out of $3!$ outcomes is favourable otherwise.) 
 
\begin{eqnarray*}
Pr[(e \preceq_{\sigma_i} f)~or~(f \preceq_{\sigma_i} e)] & = & Pr[(e \preceq_{\sigma_i} f)] + Pr[(f \preceq_{\sigma_i} e)] \\
& = & \frac{1}{6} + \frac{1}{6} \\
& = & \frac{1}{3}
\end{eqnarray*}
Therefore, 
\begin{eqnarray*}
Pr[\bigcap_{i=1}^r\left((e \npreceq_{\sigma_i} f) \cap (f \npreceq_{\sigma_i} e)\right)] & = & \left(Pr[(e \npreceq_{\sigma_i} f) \cap (f \npreceq_{\sigma_i} e)]\right)^r \\
& = & (1-\frac{1}{3})^r \\
& = & \left(\frac{2}{3}\right)^r 
\end{eqnarray*}
\begin{eqnarray*}
Pr[\bigcup_{\forall\mbox{ pairs of distinct edges  }e,f}\left(\bigcap_{i=1}^r\left((e \nprec_{\sigma_i} f) \cap (f \nprec_{\sigma_i} e)\right)\right)] < n^4 \left(\frac{2}{3}\right)^r 
\end{eqnarray*}
Substituting for $r= 6.84\log n$ in the above inequality, we get 
\begin{eqnarray*}
Pr[\bigcup_{\forall\mbox{ pairs of distinct edges  }e,f}\left(\bigcap_{i=1}^r\left((e \nprec_{\sigma_i} f) \cap (f \nprec_{\sigma_i} e)\right)\right)] < 1
\end{eqnarray*}
That is, there exists a family of permutations of $V(K_n)$ of cardinality at most $6.84\log n$ which is pairwise suitable and $3$ mixing for $K_n$.  
\qed \end{proof}

\subsubsection*{Tightness of Theorem \ref{theoremBoxliSize}}
Let $K_n$ denote a complete graph on $n$ vertices. Since $\omega(K_n) = n$, it follows from Corollary \ref{corollaryBoxliOmega} that $\boxli(K_n) \geq \log \floor{n/2}$. Hence the bound proved in Theorem \ref{theoremBoxliSize} is tight up to a constant factor.  

\subsection{Acyclic and star chromatic number}
\label{subsection:Acyclic}
\begin{definition}
\label{definitionAcyclicStarChromatic}
The \emph{acyclic chromatic number} of a graph $G$, denoted by $\chi_a(G)$, is the minimum number of colours needed to do a proper colouring of  the vertices of $G$ such that the graph induced on the vertices of every pair of colour classes is acyclic. The \emph{star chromatic number} of a graph $G$, denoted by $\chi_s(G)$, is the minimum number of colours needed to do a proper colouring of  the vertices of $G$ such that the graph induced on the vertices of every pair of colour classes is a star forest. 
\end{definition}

We know that that a star forest is a disjoint union of stars. Therefore, $\chi_s(G) \geq \chi_a(G) \geq \chi(G)$, where $\chi(G)$ denotes the chromatic number of $G$. In order to bound $\boxli(G)$ in terms of $\chi_a(G)$ and $\chi_s(G)$, we first bound $\boxli(G)$ for forests and star forests. Then the required result follows from an application of Lemma \ref{lemmaMaxPairsParts}.

\begin{lemma}
\label{lemmaBoxliForests}
For a forest $G$, $\boxli(G) \leq 2$.  
\end{lemma}
\begin{proof}
Proof in Appendix \ref{AppSection:lemmaBoxliForests}
\qed \end{proof}

\begin{lemma}
\label{lemmaBoxliStars}
For a star forest $G$, $\boxli(G) = 1$.  
\end{lemma}
\begin{proof}
Follows directly from Theorem \ref{prop:caterpillar}. 
%Let $S_1, \ldots , S_r$ be the collection of stars that form $G$. Let $\sigma$ be a permutation of $V(G)$ which satisfies $V(S_1) \prec_{\sigma} \cdots \prec_{\sigma} V(S_r)$. It is easy to verify that $\{\sigma\}$ is pairwise suitable for $G$. 
\qed \end{proof}

Using Theorem \ref{theoremBoxliSize}, we shall now prove the following lemma which together with Lemmas \ref{lemmaBoxliForests} and \ref{lemmaBoxliStars} will give us Theorem \ref{theoremBoxliAcyclicStar}. 

\begin{lemma}
\label{lemmaMaxPairsParts}
Let $P_G=\{V_1, \ldots, V_r\}$ be a partitioning of the vertices of a graph $G$, i.e., $V(G) = V_1 \uplus \cdots \uplus V_r$. Let $\hat{\boxli}(P_G) = \max_{i,j \in [r]} \boxli(G[V_i \cup V_j])$. Then, $\boxli(G) \leq 13.68 \log r + \hat{\boxli}(P_G) r$.   
\end{lemma}
\begin{proof}
Proof in Appendix \ref{AppSection:lemmaMaxPairsParts}
\qed \end{proof}

\begin{theorem}
\label{theoremBoxliAcyclicStar}
For a graph $G$, $\boxli(G) \leq 2\chi_a(G) + 13.68\log(\chi_a(G))$. Further, if the star chromatic number of $G$ is $\chi_s$, then   $\boxli(G) \leq \chi_s(G) + 13.68\log(\chi_s(G))$. 
\end{theorem}
\begin{proof}
The theorem follows directly from Lemmas \ref{lemmaBoxliForests}, \ref{lemmaBoxliStars}, and \ref{lemmaMaxPairsParts}.  
\qed \end{proof}

This, together with some existing results from literature, gives us a few easy corollaries. Alon, Mohar, and Sanders have showed that a graph embeddable in a surface of Euler genus $g$ has an acyclic chromatic number in $O(g^{4/7})$ \cite{alonacyclicgenus}. It is noted by Esperet and Joret in \cite{esperet2011boxicity}, using results of Nesetril, Ossona de Mendez, Kostochka, and Thomason, that graphs with no $K_t$ minor have an acyclic chromatic number in $\order{t^2 \log t}$. Hence the following corollary.

\begin{corollary}
\label{corollaryBoxliGenus}
\begin{enumerate}[(i)]
\item For a graph $G$ with Euler genus $g$, $\boxli(G) \in O(g^{4/7})$; and
\item for a graph $G$ with no $K_t$ minor, $\boxli(G) \in O(t^2 \log t)$.  
\end{enumerate}
\end{corollary}

% It is well known that $\chi_a(G) \leq tw(G) + 1$ for every graph $G$. Hence $\boxli(G) \in \order{tw(G)}$. The dependence on treewidth is much weaker in this upper bound compared to the one in \ref{corollaryBoxliTreewidth}, but this one has no dependence on $n$.

\subsection{Planar graphs}
\label{subsection:planar}
Since planar graphs have acyclic chromatic number at most $5$ \cite{borodin1979acyclic}, it follows from Theorem \ref{theoremBoxliAcyclicStar} that, for every planar graph $G$, $\boxli(G) \leq 42$. Using Schnyder's celebrated result on non-crossing straight line plane drawings of planar graphs we improve this bound to the best possible.

\begin{theorem}[Schnyder, Theorem $1.1$ in \cite{schnyder1990embedding}]
\label{theoremSchnyder}
Let $\lambda_1$, $\lambda_2$, $\lambda_3$ be three pairwise non parallel straight lines in the plane. Then, each plane graph has a straight line embedding in which any two disjoint edges are separated by a straight line parallel to $\lambda_1$, $\lambda_2$ or $\lambda_3$. 
\end{theorem} 

This immediately gives us the following tight bound for planar graphs.

\begin{theorem}
\label{theoremBoxliPlanar}
Separation dimension of a planar graph is at most $3$. More over there exist planar graphs with separation dimension $3$.
\end{theorem}
\begin{proof}
Consider the following three pairwise non parallel lines in $\R^2$: $\lambda_1 = \{(x,y) : y = 0, x \in \R \} $, $\lambda_2 = \{(x,y) : x = 0, y \in \R \}$ and $\lambda_3 = \{(x,y): x,y \in \R, x+y = 0\}$. Let $f: V(G) \into \R^2$ be  an embedding  such that any two disjoint edges in $G$ are separated by a straight line parallel to $\lambda_1$, $\lambda_2$ or $\lambda_3$. For every vertex $v$, let $v_x$ and $v_y$ denote the projections of $f(v)$ on to the $x$ and $y$ axes respectively. 

Construct $3$ permutations $\sigma_1, \sigma_2, \sigma_3$ such that $u_x < v_x \implies u \prec_{\sigma_1} v$, $u_y < v_y \implies u \prec_{\sigma_2} v$, and $u_x + u_y < v_x + v_y \implies u \prec_{\sigma_3} v$, with ties broken arbitrarily. Now it is easy to verify that any two disjoint edges of $G$ separated by a straight line parallel  to $\lambda_i$ in the embedding $f$, will be separated in $\sigma_i$.

Tightness of the theorem follows from considering $K_4$, the complete graph on $4$ vertices which is a planar graph. Any single permutation of its $4$ vertices separates exactly one pair of disjoint edges. Since $K_4$ has $3$ pairs of disjoint edges, we need exactly $3$ permutations.
\qed \end{proof}

\subsubsection*{Outerplanar and series-parallel graphs}
We know that outerplanar graphs form a subclass of series-parallel graphs which in turn form a subclass of planar graphs. It is not difficult to see that the separation dimension of outerplanar graphs is at most $2$. The idea is to take one permutation by reading the vertices from left to right along the spine in a one page embedding of the graph and the second permutation in the order in which we see the vertices when we recursively peel off the outermost edge till every vertex is enlisted. As for series-parallel graphs, we show that there exist series-parallel graphs with separation dimension $3$ (see Appendix \ref{AppSection:prop:seriesparallel}). 
%The graph $S$ given in Figure \ref{graphS} is an example of a series-parallel graph of separation dimension $3$. 
%\begin{proof}
%Proof in Appendix \ref{AppSection:prop:seriesparallel}
%\qed \end{proof}

\subsection{Lower bounds}
\label{subsection:LowerBound}
The tightness of many of the upper bounds we showed in the previous section relies on the lower bounds we derive in this section. First, we show that if a graph contains a uniform bipartite subgraph, then it needs a large separation dimension. This immediately gives a lower bound on separation dimension for complete bipartite graphs and hence a lower bound for every graph $G$ in terms $\omega(G)$. The same is used to obtain a lower bound on the separation dimension for random graphs of all density. Finally, it is used as a critical ingredient in proving a lower bound on the separation dimension for complete $r$-uniform hypergraphs. 
%Before we close this section we give a lower bound on the separation dimension of $K_n^{1/2}$ using Erd\H{o}s-Szekeres Theorem and a lower bound on the poset dimension of canonical interval orders.

%%%%% Subsections under lower bounds begins. %%%%%%%%%%%%%%%%%%%%%%%%%%%%%%%%%%%%%%%%%%%%%%%%%%%%%%

%\subsection{Uniform bipartitions}
%\label{sectionLowerBoundBipartition}

\begin{theorem}
\label{theoremBoxliLowerBound}
For a graph $G$, let $V_1, V_2 \subsetneq V(G)$ such that $V_1 \cap V_2 = \emptyset$. If  there exists an edge between every $s_1$-subset of $V_1$ and every $s_2$-subset of $V_2$, then $\boxli(G) \geq \min \left\{ \log \frac{|V_1|}{s_1}, \log \frac{|V_2|}{s_2} \right\}$.  
\end{theorem}
\begin{proof}
Let $\F$ be a family of permutations of $V(G)$ that is pairwise suitable for $G$. Let $r = |\F|$. We claim that, for any $\sigma \in \F$, there always exists an $S_1 \subseteq V_1$ and an $S_2 \subseteq V_2$ such that $|S_1| \geq \ceil{ |V_1|/2 }, |S_2| \geq \ceil{ |V_2|/2 }$ and $S_1 \prec_{\sigma} S_2$ or $S_2 \prec_{\sigma} S_1$. To see this, scan $V(G)$ in the order of $\sigma$ till we see $\ceil{|V_1|/2 }$ elements from $V_1$ or $\ceil{|V_2|/2}$ elements of $V_2$, which ever happens earlier. In the former case the first $\ceil {|V_1|/2 }$ elements of $V_1$ precede at  least $\ceil{ |V_2|/2 }$ elements of $V_2$ and in the latter case the first $\ceil{|V_2|/2}$ elements of $V_2$ precede at least $\ceil{|V_1|/2}$ elements of $V_1$. Extending this claim recursively to all permutations in $\F$, we see that there always exist a $T_1 \subseteq V_1$ and a $T_2 \subseteq V_2$ such that $|T_1| \geq |V_1|/2^r, |T_2| \geq |V_2|/2^r$ and $\forall \sigma \in \F$, either $T_1 \prec_{\sigma} T_2$ or $T_2 \prec_{\sigma} T_1$. We now claim that either $|T_1| \leq s_1$ or  $|T_2| \leq s_2$. Suppose, for contradiction, $|T_1| \geq s_1+1$ and  $|T_2| \geq s_2+1$. Then by the statement of the theorem, there exists an edge $e = \{v_1, v_2\}$ of $G$ such that $v_1 \in T_1$ and $v_2 \in T_2$ and a second edge $f$ between $T_1 \setminus \{v_1\}$ and $T_2 \setminus \{v_2\}$. Since $T_1$ and $T_2$ are separated in every permutation of $\F$, no permutation in $\F$ separates the disjoint edges $e$ and $f$ between $T_1$ and $T_2$. This contradicts the fact that $\F$ is a pairwise suitable family for $G$. Hence, either $|V_1| / 2^r \leq |T_1| \leq s_1$ or $|V_2|/2^r \leq |T_2| \leq s_2$ or both. That is, $r \geq \min \left\{ \log \frac{|V_1|}{s_1}, \log \frac{|V_2|}{s_2} \right\}$. 
\qed \end{proof}

The next two corollaries are immediate. 

\begin{corollary}
\label{corollaryCompleteBipartiteLowerBound}
For a complete bipartite graph $K_{m,n}$ with $m \leq n$, $\boxli(K_{m,n}) \geq \log(m)$. 
\end{corollary}
\begin{corollary}
\label{corollaryBoxliOmega}
For a graph $G$, 
$$\boxli(G) \geq \log \floor{\frac{\omega(G)}{2}},$$
where $\omega(G)$ is the size of a largest clique in $G$.
\end{corollary}

%%%% Dense bipartite lower bound subsection ends. %%%%%%%%%%%%%%%%%%%%%%%%%%%%%%%%%%%%%%%%%%%%%%%%%

\subsection{Random graphs}
\label{sectionRandomGraphs}

% In this section we use Theorem \ref{theoremBoxliLowerBound} to obtain a lower bound on $\boxli(G)$ of a random graph $G \in \G(n,p)$. 

\begin{definition}[Erd\H{o}s-R\'{e}nyi model]
$\G(n,p)$, $n \in \N$ and $0 \leq p \leq 1$, is the discrete probability space of all simple undirected graphs $G$ on $n$ vertices with each pair of vertices of $G$ being joined by an edge with a probability $p$ independent of the choice for every other pair of vertices. 
\end{definition}

\begin{definition}
A property $P$ is said to hold for $\G(n,p)$ {\em asymptotically almost surely (a.a.s)} if the probability that $P$ holds for $G \in \G(n,p)$ tends to $1$ as $n$ tends to $\infty$.
\end{definition}

%%%%%%%%%%%%% An attempt to improve.

\begin{theorem}
\label{theoremBoxliLowerBoundRandom}
For $G \in \G(n,p(n))$ % with $p(n) > \frac{e^{e/4}}{n}$, 
$$\boxli(G) \geq \log(np(n)) - \log\log(np(n)) - 2.5 \mbox{ a.a.s}.$$
\end{theorem}
\begin{proof}
Proof in Appendix \ref{AppSection:theoremBoxliLowerBoundRandom}. 
\qed \end{proof}
Note that the expected average degree of a graph in $\G(n,p)$ is $\mathbb{E}_p[\bar{d}] = (n-1)p$. And hence the above bound can be written as $\log \mathbb{E}_p[\bar{d}] - \log\log \mathbb{E}_p[\bar{d}] - 2.5$.

\section{Separation dimension of hypergraphs}
\label{sectionHypergraphs}
\subsection{Separation dimension and size of a hypergraph}
\label{sectionBoxliSizeHyper}

% Next, we study pairwise suitable family of permutations for hypergraphs of rank $r$. A probabilistic argument similar to the one used in Theorem \ref{theoremBoxliSize} gives the next result. 
Using a direct probabilistic argument we obtain the following theorem.  
\begin{theorem}
\label{theoremHypergraphSizeUpperbound}
For any rank-$r$ hypergraph $H$ on $n$ vertices 
$$ \boxli(H) \leq \frac{e \ln 2}{\pi \sqrt{2}} 4^r \sqrt{r} \log n.$$
\end{theorem}
\begin{proof}
Proof in Appendix \ref{AppSection:theoremHypergraphSizeUpperbound}
\qed \end{proof}
%\deepak{
%The non-edges can be selected in the following way. The first vertex of the edge, which exists, can be selected in $n$ ways. The second to $r$-%th element can be chosen to be one of the vertices not chosen so far or can be left unfilled. But then the maximum possibilities per choice %is sill at most $n$. Hence the bound. (Should we write this argument?)
%}

\subsubsection*{Tightness of Theorem \ref{theoremHypergraphSizeUpperbound}}

Let $K_n^r$ denote a complete $r$-uniform graph on $n$ vertices. Then by Theorem \ref{theoremHypergraphSizeLowerbound}, $\boxli(K_n^r) \geq \frac{1}{2^7} \frac{4^r}{\sqrt{r-2}} \log n$ for $n$ sufficiently larger than $r$. Hence the bound in Theorem \ref{theoremHypergraphSizeUpperbound} is tight by factor of $64 r$.

\subsection{Maximum Degree}
\label{Section:HyperGraphMaxDeg}
\begin{theorem}
For any rank-$r$ hypergraph $H$ of maximum degree $D$,
$\boxli(H) \leq \order{rD \log^2(rD)}$. 
\end{theorem}
\begin{proof}
This is a direct consequence of the nontrivial fact that $\boxi(G) \in \order{\Delta \log^2 \Delta}$ for any graph $G$ of maximum degree $\Delta$ \cite{DiptAdiga}. 
It is known that there exist graphs of maximum degree $\Delta$ whose boxicity can be as high as $c \Delta \log \Delta$ \cite{DiptAdiga}, where $c$ is a small enough positive constant. Let $G$ be one such graph. Consider the following hypergraph $H$ constructed from $G$. Let $V(H) = E(G)$ and $E(H) = \{E_v : v \in V(G)\}$ where $E_v$ is the set of edges incident on the vertex $v$ in $G$. It is clear that $G = L(H)$. Hence $\boxli(H) = \boxi(G) \geq c \Delta(G) \log \Delta(G)$. 

Note that the rank of $H$ is $r = \Delta(G)$ and the maximum degree of $H$ is $2$. Thus $\boxli(H) \geq c r \log(r)$ and hence the dependence on $r$ in the upper bound  cannot be considerably brought down in general. 
\qed
\end{proof}
% It is trivial to see that the separation dimension of hypergraphs with maximum degree $1$ cannot be more than $1$.

\subsection{Lower bound}
\label{sectionLowerBoundHypergraphs}

Now we illustrate one method of extending the above lower bounding technique from graphs to hypergraphs.
Let $K_n^r$ denote the complete $r$-uniform hypergraph on $n$ vertices. We show that the upper bound of 
$\order{4^r \sqrt{r} \log n}$ obtained for $K_n^r$ from Theorem \ref{theoremHypergraphSizeUpperbound} is 
 tight up to a factor of $r$. The lower bound argument below is motivated by an argument used by 
  Radhakrishnan to prove a lower bound on the size of a family of scrambling permutations \cite{radhakrishnan2003note}. From Corollary \ref{corollaryBoxliOmega} we know that the separation dimension 
   of $K_n$, the complete graph on $n$ vertices, is in $\orderatleast{\log n}$. Below we show that  given any
    separating embedding of $K_n^r$ in $\R^d$, the space $\R^d$ contains ${2r-4 \choose r-2}$ orthogonal 
     subspaces such that the projection of the given embedding on to these subspaces gives a separating 
      embedding
     of a $K_{n-2r+4}$. 

\begin{theorem}
\label{theoremHypergraphSizeLowerbound}
Let $K_n^r$ denote the complete $r$-uniform hypergraph on $n$ vertices with $r > 2$. Then 
$$ c_1 \frac{4^r}{\sqrt{r-2}} \log n \leq \boxli(K_n^r) \leq c_2 4^r \sqrt{r} \log n,$$ for $n$ sufficiently larger than $r$ and where $c_1 = \frac{1}{2^7}$ and  $c_2 =  \frac{e\ln2}{\pi\sqrt{2}} < \frac{1}{2}$.
\end{theorem}
\begin{proof} The upper bound follows from Theorem \ref{theoremHypergraphSizeUpperbound} and so it suffices to prove the lower bound.

\def\S{\mathcal{S}}

Let $\F$ be a family of pairwise suitable permutations for $K_n^r$. Let $\S$ be a maximal family of $(r-2)$-sized subsets of $[2r-4]$ such that if $S \in \S$, then $[2r-4] \setminus S \notin \S$. Hence $|\S| = \frac{1}{2} {2r-4 \choose r-2} \geq 2^{-6} 4^r / \sqrt{r-2}$ (using the fact that $\sqrt{k} {2k \choose k} \geq 2^{2k-1}$). Notice that for any permutation $\sigma \in \F$, if $S \in \S$ and $[2r-4] \setminus S$ are separated in $\sigma$ then no other $S' \in \S$ and $[2r-4] \setminus S'$ are separated in $\sigma$. Hence we partition $\F$ into $|\S|$ (disjoint) sub-families $\{\F_S\}_{S \in \S}$ such that $\sigma \in \F_S$ if and only if $\sigma$ separates $S$ and $[2r-4] \setminus S$. We claim that each $\F_S$ is pairwise suitable for the complete graph on the vertex set $\{2r-3, \ldots, n\}$, i.e, for any distinct $a, b, c, d \in \{2r-3, \ldots, n\}$ there exists some $\sigma \in \F_S$ which separates $\{a, b\}$ from  $\{c, d\}$. This is because the permutation $\sigma \in \F$ which separates the $r$-sets $S \cup \{a,b\}$ from $([2r-4] \setminus S) \cup \{c, d\}$ lies in $\F_S$. Hence by Corollary \ref{corollaryBoxliOmega}, we have $|\F_S| \geq \log \floor{(n - 2r +4)/2}$. Since $\F = \biguplus_{S \in \S} \F_S$, we have $|\F| \geq |\S||\F_S| \geq 2^{-6} \frac{4^r}{\sqrt{r-2}} \log \floor{ (n-2r+4)/2 }$ which is at least $2^{-7} \frac{4^r}{\sqrt{r-2}} \log n$ for $n$ sufficiently larger than $r$.
\qed \end{proof}

%%%%% Lower bound for hypergraph subsection ends. %%%%%%%%%%%%%%%%%%%%%%%%%%%%%%%%%%%%%%%%%%%%%%%%%

\section{Discussion and open problems}
\label{sectionOpenProblems}
Since $\boxli(G)$ is the boxicity of the line graph of $G$, it is interesting to see how it is related to boxicity of $G$ itself. But unlike separation dimension, boxicity is not a monotone parameter. For example the boxicity of $K_n$ is $1$, but deleting a perfect matching from $K_n$, if $n$ is even, blows up its boxicity to $n/2$. Yet we couldn't find any graph $G$ such that $\boxi(G) > 2^{\boxli(G)}$. Hence we are curious about the following question: Does there exist a function $f : \mathbb{N} \into \mathbb{N}$ such that $\boxi(G) \leq f(\boxli(G))$?
%\begin{openproblem}
%Does there exist a function $f : \mathbb{N} \into \mathbb{N}$ such that $\boxi(G) \leq f(\boxli(G))$?
%\end{openproblem}
Note that the analogous question for $\boxlistar(G)$ has an affirmative answer. If there exists a vertex $v$ of degree $d$ in $G$, then any $3$-mixing family of permutations of $V(G)$ should contain at least $\log d$ different permutations because any single permutation will leave $\ceil{d/2}$ neighbours of $v$ on the same side of $v$. Hence $\log \Delta(G) \leq \boxlistar(G)$. From \cite{DiptAdiga}, we know that $\boxi(G) \in \order{\Delta(G) \log^2 \Delta(G)}$ and hence $\boxi(G) \in \order{2^{\boxlistar(G)} (\boxlistar(G))^2}$.

Another interesting direction of enquiry is to find out the maximum number of hyperedges (edges) possible in a hypergraph (graph) $H$ on $n$ vertices with $\boxli(H) \leq k$. Such an extremal hypergraph $H$, with $\boxli(H) \leq 0$, is seen to be a maximum sized intersecting family of subsets of $[n]$. A similar question for order dimension of a graph has been studied \cite{agnarsson1999maximum,agnarsson2002extremal} and has found applications in ring theory. We can also ask a three dimensional analogue of the question answered by Schnyder's theorem in two dimensions. Given a collection $P$ of non parallel planes in $\R^3$, can we embed a graph $G$ in $\R^3$ so that every pair of disjoint edges is separated by a plane parallel to one in $P$. Then $|P|$ has  to be at least $\boxli(G)$ for this to be possible. This is because the permutations induced by projecting such an embedding onto the normals to the planes in $P$ gives a pairwise suitable family of permutations of $G$ of size $|P|$. Can $|P|$ be upper bounded by a function of $\boxli(G)$?

We know that Theorem \ref{prop:caterpillar} yields a linear time algorithm for recognizing graphs of separation dimension at most $1$. This gives rise to a very natural question. Is it possible to recognize graphs of separation dimension at most $2$ in polynomial time? 
\bibliographystyle{plain}
%\bibliography{mathewref}

\newpage
\appendix
\section{Appendix}
\label{sectionAppendix}

\subsection{Proof of Lemma \ref{lemmaBoxliForests}}
\label{AppSection:lemmaBoxliForests}
\begin{proof}
Let $T_1, \ldots , T_r$ be the collection of trees that form $G$.  Convert each tree $T_i$ to an ordered tree by arbitrarily choosing a root vertex for $T_i$ and assigning an arbitrary order to the children of each vertex. Let $\sigma_1, \sigma_2$ be two permutations of $V(G)$ defined as explained below. Consider a vertex $u \in V(T_i)$ and a vertex $v \in V(T_j)$, where $i,j \in [r]$. If $i \neq j$, then $u \prec_{\sigma_1} v \iff i < j$ and $u \prec_{\sigma_2} v \iff i < j$. Otherwise, $u \prec_{\sigma_1} v$ if and only if $u$ precedes $v$ in a preorder traversal of the ordered tree $T_i$ and $u \prec_{\sigma_2} v$ if and only if $u$ precedes $v$ in a postorder traversal of the ordered tree $T_i$. It is left to the reader to verify that $\{\sigma_1, \sigma_2\}$ form  pairwise suitable family of permutations for $G$. 
\qed \end{proof}

\subsection{Proof of Lemma \ref{lemmaMaxPairsParts}}
\label{AppSection:lemmaMaxPairsParts}
\begin{proof}
Let $H$ be a complete graph with $V(H) = \{h_1, \ldots , h_r\}$. Let $\mathcal{M} = \{M_1, \ldots ,M_r \}$ be a collection of matchings of $H$ such that each edge is present in at least one matching $M_i$. It is easy to see that there exists such a collection (Vizing's Theorem on edge colouring - Theorem $5.3.2$ in \cite{Diest}).
 For each $i \in [r]$, let $G_i$ be a subgraph of $G$ such that $V(G_i) = V(G)$ and for a pair of vertices $u \in V_a$, $v \in V_b$, $\{u,v\} \in E(G_i)$ if $a=b$ or $\{h_a,h_b\} \in M_i$. Note that $G_i$ is made of $|M_i|$ disjoint components. Let $\F_i$ be a family of permutations that is pairwise suitable for $G_i$ such that $|\F_i| = \boxli(G_i)$. By Observation \ref{observationDisjointComponents}, we have $|\F_i| \leq \hat{\boxli}(P_G)$. 

From Theorem \ref{theoremBoxliSize}, $\boxlistar(H) \leq 6.84 \log r$. Let $\E$ be a family of permutations that is pairwise suitable and $3$-mixing for $H$ such that $|\E| = \boxlistar(H) \leq 6.84 \log r$. We construct two families of permutations, namely $\F_{r+1}$ and $\F_{r+2}$, of $V(G)$ from $\E$ such that $|\F_{r+1}| = |\F_{r+2}| = |\E|$. Corresponding to each permutation $\sigma \in \E$, we construct $\tau_{\sigma} \in \F_{r+1}$ and $\kappa_{\sigma} \in \F_{r+2}$ as follows. If $h_i \prec_{\sigma} h_j$, then we have $V_i \prec_{\tau_{\sigma}} V_j$ and $V_i \prec_{\kappa_{\sigma}} V_j$. Moreover, for each $i \in [r]$ and for distinct $v,v' \in V_i$, $v \prec_{\tau_{\sigma}} v'\iff v' \prec_{\kappa_{\sigma}} v$. 
\begin{claim}
\label{claimMaxPairsParts}
$\F = \bigcup_{i=1}^{r+2}\F_i$ is a pairwise-suitable family of permutations for $G$. 
\end{claim}

We prove the claim by showing that for every pair of non-adjacent edges $e, e' \in E(G)$, there is a $\sigma \in \F$ such that $e \prec_{\sigma} e'$ or $e' \prec_{\sigma} e$. We call an edge $e$ in $G$ a \emph{crossing edge} if there exists distinct $i,j \in [r]$  such that $e$ has its endpoints in $V_i$ and $V_j$. Otherwise $e$ is called a \emph{non-crossing edge}. Consider any two disjoint edges $\{a,b\}, \{c,d\}$ in $G$. Let $a \in V_i, b \in V_j, c \in V_k$ and $d \in V_l$. If $|\{i,j,k,l\}| \leq 2$, then both the edges belong to some $G_p, p \in [r]$ and hence are separated by a permutation in $\F_p$. If $|\{i,j,k,l\}| = 3$, then the two edges are separated by a permutation in $\F_{r+1}$ or $\F_{r+2}$ since $\E$ was $3$-mixing for $H$. If $|\{i,j,k,l\}| = 4$, then the two edges are separated by a permutation in both $\F_{r+1}$ and $\F_{r+2}$ since $\E$ was pairwise suitable for $H$. Details follow.

\setcounter{case}{0}
\begin{case}[both $\{a,b\}$ and $\{c,d\}$ are crossing edges]

If $i,j,k$ and $l$ are distinct then from the definition of $\E$ there exists a permutation $\sigma \in \E$ such that $\{h_i,h_j\} \prec_{\sigma} \{h_k, h_l\}$ or $\{h_k, h_l\} \prec_{\sigma} \{h_i,h_j\}$. Without loss of generality, assume $\{h_i,h_j\} \prec_{\sigma} \{h_k, h_l\}$. Therefore, in the permutations $\tau_{\sigma}$ and $\kappa_{\sigma}$ constructed from $\sigma$, we have $\{a,b\} \prec_{\tau_{\sigma}} \{c,d\}$ and $\{a,b\} \prec_{\kappa_{\sigma}} \{c,d\}$. 

Recall that $\E$ is a pairwise suitable and $3$-mixing family of permutations for $H$. If $i=k$ and $i,j,l$ are distinct, then there exists a permutation $\sigma \in \E$ such that $h_j \prec_{\sigma} h_i \prec_{\sigma} h_l$ or $h_l \prec_{\sigma} h_i \prec_{\sigma} h_j$. Without loss of generality, assume  $h_j \prec_{\sigma} h_i \prec_{\sigma} h_l$. Now it is easy to see that either $\{a,b\} \prec_{\tau_{\sigma}} \{c,d\}$ or $\{a,b\} \prec_{\kappa_{\sigma}} \{c,d\}$. The cases when $i=l, j, k$ are distinct or $i,j=k,l$ are distinct or $i,j=l,k$ are distinct are symmetric to the above case where  $i=k,j,l$ are distinct. 

Consider the case when $i=k,j=l$ are distinct. In this case, both $\{a,b\}$ and $\{c,d\}$ have their endpoints in $V_i$ and $V_j$. Then there exists some $p \in [r]$ such that $\{a,b\}, \{c,d\} \in E(G_p)$. Since $\F_p$ is a pairwise suitable family of permutations for $G_p$ there exists a $\sigma \in F_p$ such that $\{a,b\} \prec_{\sigma} \{c,d\}$ or  $\{c,d\} \prec_{\sigma} \{a,b\}$. The case when $i=l$ and $j=k$ are distinct is similar. 
\end{case}

\begin{case}[only $\{a,b\}$ is a crossing edge] 

Let $a \in V_i, b \in V_j$ and $c,d \in V_k$. If $i,j,k$ are distinct then there exists a permutation $\sigma$ in $\E$ such that either $h_i \prec_{\sigma} h_j \prec_{\sigma} h_k$ or $h_k \prec_{\sigma} h_j \prec_{\sigma} h_i$. Without loss of generality, assume $h_i \prec_{\sigma} h_j \prec_{\sigma} h_k$. Now its easy to see that both $\{a,b\} \prec_{\tau_{\sigma}} \{c,d\}$ and $\{a,b\} \prec_{\kappa_{\sigma}} \{c,d\}$. If $i=k, j$ are distinct then both $\{a,b\}$ and $\{c,d\}$ have their endpoints from $V_i \cup V_j$. Then there exists some $p \in [r]$ such that $\{a,b\}, \{c,d\} \in E(G_p)$. Since $\F_p$ is a pairwise suitable family of permutations for $G_p$ there exists a $\sigma \in \F_p$ such that $\{a,b\} \prec_{\sigma} \{c,d\}$ or  $\{c,d\} \prec_{\sigma} \{a,b\}$. The case when $j=k, i$ are distinct is similar.   
\end{case}

\begin{case}[only $\{c,d\}$ is a crossing edge] 
Similar to the case above. 
\end{case}

\begin{case}[both $\{a,b\}$ and $\{c,d\}$ are non-crossing edges]
Then, for each $p \in [r]$,  $\{a,b\}, \{c,d\} \in E(G_p)$. Since $\F_p$ is a pairwise suitable family of permutations for $G_p$ there exists a $\sigma \in \F_p$ such that $\{a,b\} \prec_{\sigma} \{c,d\}$ or  $\{c,d\} \prec_{\sigma} \{a,b\}$. 
\end{case}

Thus, we prove Claim \ref{claimMaxPairsParts}. Hence, we have $\boxli(G) \leq |\mathcal{F}| = \sum_{i=1}^{r}|F_i| + |F_{r+1}| + |F_{r+2}| \leq \hat{\boxli}(P_G) r + 13.68 \log r$.
\qed \end{proof}

\subsection{Series-parallel graphs of separation dimension $3$}
\label{AppSection:prop:seriesparallel}

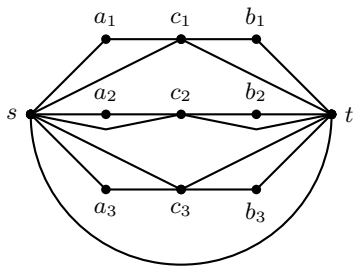
\begin{figure}[!ht]
\begin{center}
\begin{pspicture}(-0.5,-0.5)(4.5,3.5)
% 	\psgrid
	\psline[showpoints=true](0,1)(1,2)(2,2)(3,2)(4,1)
	\psline[showpoints=true](0,1)(1,1)(2,1)(3,1)(4,1)
	\psline[showpoints=true](0,1)(1,0)(2,0)(3,0)(4,1)
	\psline[showpoints=true](0,1)(2,0)(4,1)
	\psline[showpoints=false](0,1)(1,0.8)(2,1)
	\psline[showpoints=false](2,1)(3,0.8)(4,1)
	\psline[showpoints=true](0,1)(2,2)(4,1)
%	\psline[showpoints=false](0,1)(2,0.5)(4,1)
	\psarc[showpoints=false](2,1){2}{180}{0}
%	\psline[showpoints=false](0,1)(0.3,2)(1,2.8)(2,3)(3,2.8)(3.7,2)(4,1)

	\uput[l](0,1){$s$}
	\uput[90](1,2){$a_1$}
	\uput[90](2,2){$c_1$}
	\uput[90](3,2){$b_1$}
	\uput[90](1,1){$a_2$}
	\uput[90](2,1){$c_2$}
	\uput[90](3,1){$b_2$}
	\uput[270](1,0){$a_3$}
	\uput[270](2,0){$c_3$}
	\uput[270](3,0){$b_3$}
	\uput[r](4,1){$t$}
	
\end{pspicture}
\end{center}
\caption{Series-parallel graph $S$ of separation dimension $3$}
\label{figureGraphS}
\end{figure}

\begin{Proposition}
\label{prop:seriesparallel}
For graph $S$ in Figure \ref{figureGraphS}, $\boxli(S) = 3$. 
\end{Proposition}
\begin{proof}
Since series-parallel graphs are planar graphs, we know from Theorem \ref{theoremBoxliPlanar} that $\boxli(S) \leq 3$. Assume for contradiction that $\boxli(G) \leq 2$. 
%Let $\sigma_1$ and $\sigma_2$ be two permutations of $V(S)$ that is pairwise suitable for $S$. 
We claim that no two distinct vertices $c_i$ and $c_j$ (where $1 \leq i< j \leq 3$) can together succeed (or together precede) both $s$ and $t$ in any of the permutations. This is because in such a situation the disjoint edges of the diamond graph $D$ (i.e. the graph with $4$ vertices and $5$ edges) induced on $s,~ c_i, ~c_j$ and $t$ will have to be separated in the one remaining permutation which we know is impossible (as $\boxli(D) = 2$). We also claim that every $c_i$ has to succeed (or precede) both $s$ and $t$ in one of the two permutations so as to separate the $\{a_i, c_i\}$ edge from $\{s,t\}$. Since we have only two permutations, by applying the two claims proved above, we can conclude that there exists one permutation, say $\sigma$, such that $c_i \prec_{\sigma}  s \prec_{\sigma} t  \prec_{\sigma} c_j$ or $c_i \prec_{\sigma} t \prec_{\sigma} s  \prec_{\sigma} c_j$, where $i,j \in \{1,2,3\}$ and $i \neq j$. Assume $c_i  \prec_{\sigma} s \prec_{\sigma} t  \prec_{\sigma} c_j$ (the proof is similar in the other case). Let $k \in \{1,2,3\}$ such that $k \neq i$ and $k \neq j$. Note that in $\sigma$, $\{s, c_j\}$ is not separated from $\{t, c_k\}$ and $\{t,c_i\}$ is not separated from $\{s,c_k\}$. From the second claim, we know that either $c_k$ succeeds both $s$ and $t$ in the second permutation or it precedes both $s$ and $t$. In the former case $\{s,c_k\}$ is not separated from $\{t,c_i\}$ in both the permutations. In the latter case $\{t,c_k\}$ is not separated from $\{s,c_j\}$ in both the permutations. 
\qed \end{proof}

\subsection{Proof of Theorem \ref{theoremBoxliLowerBoundRandom}}
\label{AppSection:theoremBoxliLowerBoundRandom}
\begin{proof}

If $np(n) \leq e^{e/4}$, then  $\log(np(n)) - \log\log(np(n)) - 2.5 \leq 0$, and hence the statement is trivially true. So we can assume that $p(n) > e^{e/4} / n$. 

Let $s(n) = 2 \ln (np(n)) / p(n)$. Since $p(n) > e^{e/4} / n$ by assumption, $\ln (np(n)) > e/4$ and hence if $\lim_{n \tends \infty} p(n) = 0$, we get $\lim_{n \tends \infty} s(n) = \infty$. Otherwise, that is when $\liminf_{n \tends \infty} p(n) > 0$, we have $s(n) \geq 2 \ln(np(n)) / 1$ which tends to $\infty$ as $n \tends \infty$. Hence in every case $\lim_{n \tends \infty} s(n) = \infty$.

Let $V(G) = V_1 \uplus V_2$ be a balanced partition of $V(G)$, i.e., $V_1 \cap V_2 = \emptyset$ and $|V_1|, |V_2| \geq \floor{ n/2  }$. $S_1 \subseteq V_1$ and $S_2 \subseteq V_2$ be such that $|S_1| = |S_2| = s(n)$. The probability that there is no edge in $G$ between $S_1$ and $S_2$ is $(1-p(n))^{s(n)^2} \leq \exp(-p(n)s(n)^2)$. Hence the probability $q(n)$ that there exists an $s(n)$-sized set from $V_1$ and one $s(n)$-sized set from $V_2$ with no edge between them is bounded above by ${n/2 \choose s(n)}^2 \exp(-p(n)s(n)^2)$. Hence using the bound ${n \choose k} \leq (ne/k)^k$, we get
\begin{eqnarray*}
q(n) 	& \leq 	& 	\left( \frac{ne}{2s(n)} \right)^{2s(n)} \exp( -p(n)s(n)^2) \\ 
		& = 	&	\exp \left(2s(n) \ln \left( \frac{ne}{2s(n)} \right) - p(n)s(n)^2 \right) \\
		& = 	&	\exp \left(s(n) \left( 2\ln \left( \frac{np(n)e}{4\ln(np(n))}\right)  - 2 \ln(np(n)) \right) \right) \\
		& = 	&	\exp \left(s(n) \left( 2\ln \frac{e}{4}  - 2 \ln\ln(np(n)) \right) \right) \\
		& = 	&	\exp \left(-2s(n)\left(\ln\ln(np(n)) - \ln \frac{e}{4} \right) \right) \\
\end{eqnarray*}

Since $p(n) > e^{e/4} / n$, $\ln\ln(np(n)) > \ln(e/4)$ and since $\lim_{n \tends \infty} s(n) = \infty$, we conclude that  $\lim_{n \tends \infty} q(n) = 0$.

With probability $1 - q(n)$, every pair of subsets from $V_1 \times V_2$ each of size $s(n)$ has at least one edge between them. So by Theorem \ref{theoremBoxliLowerBound}, $\boxli(G) \geq \log \floor{n/2s(n)} \geq \log(np(n)) - \log\log (np(n)) - 2.5$ with probability $1 - q(n)$. Hence the theorem.
\qed \end{proof}

\subsection{Proof of Theorem \ref{theoremHypergraphSizeUpperbound}}
\label{AppSection:theoremHypergraphSizeUpperbound}
\begin{proof}
Consider family $\F$ of $m$ permutations of $[n]$ chosen independently and uniformly from the $n!$ possible ones. For an arbitrary pair of disjoint edges $e, f \in E(H)$, the probability $q$ that $e$ and $f$ are separated in $\sigma$ is at least $2 (r!)^2 / (2r)!$. Using Stirling's bounds $\sqrt{2\pi} k^{k + 1/2} e^{-k} \leq k! \leq e k^{k + 1/2} e^{-k}$, we get $q \geq \frac{2\pi\sqrt{2}}{e} \sqrt{r} / {4^r}$. The probability of the (bad) event that $e$ and $f$ are not separated in any of the $m$ permutations in $\F$ is at most $(1 - q)^m$. Since the number of non-empty edges in $H$ is less than $n^r$, by the union bound, the probability $p$ that there exists some pair of edges which is not separated in any of the permutations in $\F$ is less than $n^{2r}(1-q)^r \leq e^{2r \ln n}e^{-qm}$. Hence if $2r \ln n \leq qm$, then $p < 1$ and there will exist some  family $\F$ of size $m$ such that every pair of edges is separated by some permutation in $\F$. So $m \geq \frac{2r}{q} \ln n$ suffices. So $\boxli(H) \leq \frac{e}{\pi\sqrt{2}} 4^r \sqrt{r} \ln n$. 
\qed \end{proof}

\end{document}